\documentclass[12pt]{amsart}
\usepackage{latexsym}
\usepackage[margin=1in]{geometry}

\usepackage{graphicx,enumerate} 
\usepackage{color} 
\usepackage{amsmath, amsthm, amssymb}
\usepackage{enumerate}
\usepackage{float}
  \usepackage{hyperref}
\usepackage{stackrel}
\usepackage{mathrsfs,dsfont}
\usepackage{caption}
\usepackage{subcaption}
\usepackage{wrapfig}
\usepackage{bbm}
\usepackage{bm}
\usepackage{tikz}
\usepackage{pgfplots} %plot graphs
\usepackage{booktabs}
\usepackage{marginnote}

\numberwithin{equation}{section}

\newtheorem{theorem}{Theorem}[section]

\newtheorem{proposition}[theorem]{Proposition}
\newtheorem{lemma}[theorem]{Lemma}

\newtheorem{conjecture}[theorem]{Conjecture}

\theoremstyle{definition}
\newtheorem{definition}[theorem]{Definition}
\newtheorem{example}[theorem]{Example}

\theoremstyle{remark}
\newtheorem{remark}[theorem]{Remark}

\DeclareMathOperator{\im}{im}
\newcommand{\St}{{S}}
\newcommand{\invtPoly}{\mathcal{P}}

\newcommand{\bd}{box diagram}

\newcommand{\R}{\mathbb{R}}

\newcommand{\Rnn}{\mathbb{R}_{\geq 0}}

\def\lra{\leftrightarrows}

\def\la{\leftarrow}

\def\been{\begin{enumerate}}
\def\enen{\end{enumerate}}

\def\SS{\mathcal S}
\def\CC{\mathcal C}
\def\RR{\mathcal R}

\newcommand{\lradot}{%
  \mathrel{\ooalign{\hfil$\vcenter{
   \hbox{$\mkern3mu\scriptscriptstyle\bullet$}}$\hfil\cr$\longleftrightarrow$\cr}
  }%
}

\definecolor{DarkGreen}{rgb}{0,0.65,0}

\definecolor{NiceBlue}{rgb}{0.2,0.2,0.75}
\newcommand{\struc}[1]{{\color{NiceBlue} #1}}

\newcommand{\C}{\mathbb{C}}
\newcommand{\N}{\mathbb{N}}

\newcommand{\lam}{\lambda}
\newcommand{\epsi}{\varepsilon}

\newcommand\cV{{\ensuremath{\mathcal{V}}}}

\DeclareMathOperator{\dist}{dist}
\DeclareMathOperator{\Sym}{Sym}

\begin{document}

%*******************************************************************
%TITLE OF THE ARTICLE
%*******************************************************************
\title[Nondegenerate multistationarity in small reaction networks]{Nondegenerate multistationarity in \\ small reaction networks}

% for arXiv
\date{1 June 2018}
%\date{\today}

\author{Anne Shiu}
\address{Anne Shiu, Texas A\&M University, Department of Mathematics, Mailstop 3368, College Station TX 77843-3368 USA \medskip}

\email{annejls@math.tamu.edu}

\author{Timo de Wolff}

\address{Timo de Wolff, Technische Universit\"at Berlin, Institut f\"ur Mathematik, Stra{\ss}e des 17.~Juni 136, 10623 Berlin,
 Germany\medskip}

\email{dewolff@math.tu-berlin.de}

%*******************************************************************
%ABSTRACT
%*******************************************************************
\begin{abstract}
Much attention has been focused in recent years 
on the following algebraic problem 
arising from applications: which chemical reaction networks, when taken with mass-action kinetics,
admit multiple positive steady states?
The interest behind this question is in steady states that are stable.
As a step toward this difficult question,
here we address the question of multiple {\em nondegenerate} positive steady states.
Mathematically, this asks whether certain families of parametrized, real, sparse polynomial systems
ever admit multiple positive real roots that are simple.  
Our main results settle this problem for certain types of small networks, and our
techniques point the way forward for larger networks.

%*******************************************************************
%KEYWORDS
%*******************************************************************
\vskip .1in
\noindent {\bf Keywords:} chemical reaction network, discriminant, mass-action kinetics, multiple steady states, multiplicity of roots %, bistability, Newton polytope

%%*******************************************************************
%%AMS SUBJECT CLASSIFICATION
%%*******************************************************************
\vskip .1in
\noindent {\bf Mathematics Subject Classification:} 
37C10, % Vector fields, flows, ordinary differential equations
37C25, %Fixed points, periodic points, fixed-point index theory; 
12D10, % Polynomials: location of zeros (algebraic theorems)
14P05, % Real algebraic sets
34A34, %Nonlinear equations and systems, general
65H04, %Roots of polynomial equations; 
80A30 % Chemical kinetics, under "Classical thermodynamics, heat transfer"
\end{abstract}

%\vskip .1in

\maketitle

\section{Introduction}

This work is motivated by the \textit{Nondegeneracy Conjecture} from the study of reaction systems~\cite{Joshi:Shiu:Multistationary}:
if a reaction network admits multiple positive steady states, does it also admit multiple {\em nondegenerate} positive steady states?  
Equivalently, for certain families of parametrized sparse-polynomial systems, if one member of the family admits multiple positive roots, does some member admit multiple {\em multiplicity-one} positive roots?
In fact, there has been a great deal of work on characterizing when a network is multistationary (surveyed in~\cite{mss-review}),
but much less on nondegenerate multistationarity or the stronger condition of {\em bistability}~\cite{perspective}.
If the Nondegeneracy Conjecture is true, then the concepts 
of multistationarity and nondegenerate multistationarity 
are essentially equivalent.  
These questions are important in applications, because bistable networks are thought to underlie biochemical switches and other memory-encoding behavior~\cite{HowSwitch}.

Our main results 
verify the Nondegeneracy Conjecture for small networks
(Theorems~\ref{thm:main} and~\ref{thm:main2}).  
Namely, we replace ``multistationary'' by ``nondegenerately multistationary''
in the case of two species 
for the following result, which is~\cite[Theorems 5.8 and 5.12]{Joshi:Shiu:Multistationary}:
%----------------------------------
% THEOREM -- i=1, r=1
%----------------------------------
\begin{theorem}[Classification of multistationary networks with one reversible reaction and one irreversible reaction, or two reversible 
reactions~\cite{Joshi:Shiu:Multistationary}] \label{prop:i=1-r=1}
Let $G$ be a network consisting of:
\begin{itemize}
\item a reversible-reaction pair $y \lra y'$
and an irreversible reaction $\widetilde y \to \widetilde y'$ (Case 1), or  
\item two reversible-reaction pairs, $y \lra y'$ and $\widetilde y \lra \widetilde y'$ (Case 2). 
\end{itemize}
Then the following statements are equivalent:
\begin{enumerate}[(1)]
\item $G$ is multistationary. % (i.e., $\capPSS(G) \geq 2$). 
\item the reaction vectors are (nontrivial) scalar multiples of each other: 
 $y'-y = \lambda (\widetilde y' - \widetilde y)$ for some $0 \neq \lambda \in \mathbb{R}$, and, for some species $i$, the embedded network of $G$ obtained by removing all species except $i$ is:
 \begin{itemize}
 \item in Case 1, a 2-alternating network (``$\lra \, \, \to$''
or
 ``$ \leftarrow \, \, \lra$''), or 
 \item in Case 2, a 3-alternating network (``$\lra \, \, \lra$'').%
\end{itemize}
\end{enumerate}
\end{theorem}

\noindent
``Embedded'' and ``alternating'' networks are defined later (Definitions~\ref{def:emb} and~\ref{def:2-sign-change}).  

As an example, consider the network $G=\{0 \lra A+B~,~2A+B\to 3A+2B\}$.
Here two species, $A$ and $B$, are produced at the same rate (hence, $0 \to A+B$), and when they bind to each other, they are transported out of the cell ($0 \leftarrow A+B$) or, in the case of two units of $A$ and one of $B$ binding,
 they upregulate their own production ($2A+B \to 3A+2B$).  
Removing $B$ yields the network
$\{0 \lra A~,~2A \to 3A\}$, which informally has the form ``$\lra \, \, \to$''.
So, by Theorem~\ref{prop:i=1-r=1},
network $G$ is multistationary -- and 
our contribution here is to show that 
 $G$ is in fact nondegenerately multistationary (by Theorem~\ref{thm:main}).
Although we can obtain the same result 
by analyzing this network by hand, 
 we can now decide nondegenerate multistationarity quickly for this network and many others.

Indeed, our results add to the list of known results on nondegenerate multistationarity for small networks, summarized in Table~\ref{tab:results} (for details, see Section~\ref{sec:bkrd}).  Additionally, our proofs may point the way toward more results to add to the table, specifically results that elevate multistationarity to nondegenerate multistationarity.  

% TABLE (SUMMARY)
\begin{table}%[ht]
\renewcommand{\arraystretch}{1.3}
\centering
\begin{tabular}{@{}lll@{}}
\toprule
Network property & Nondegenerately multistationary? \\ %& Reference \\
\hline
Network with only 1 species ($s=1$) & 
\begin{tabular}[c]{@{}l@{}}If and only if some subnetwork is \\ \quad 2-alternating (Proposition~\ref{prop:main-prior}.1)~\cite{Joshi:Shiu:Multistationary} \end{tabular} 
\\

\begin{tabular}[c]{@{}l@{}}
Network consists of 1 reaction ($r=1$) \\ \quad or 1 reversible-reaction pair\end{tabular} & No (Proposition~\ref{prop:main-prior}.2)~\cite{Joshi:Shiu:Multistationary}\\% & \cite[Theorem 3.6]{Joshi:Shiu:Multistationary} \\

Network consists of 2 reactions ($r=2$) & See Proposition~\ref{prop:main-prior}.3~\cite{Joshi:Shiu:Multistationary} \\% & \cite[Theorem 5.2]
$r+s \leq 3$ & No (\cite[Corollary 3.8]{Joshi:Shiu:Multistationary}) \\% & \cite[Corollary 3.8]{Joshi:Shiu:Multistationary} \\

\begin{tabular}[c]{@{}l@{}}$s=2$ and 1 irreversible reaction \\ \quad  and 1 reversible-reaction pair\end{tabular}  & See Theorem~\ref{thm:main} \\
$s=2$ and  2 reversible-reaction pairs & See Theorem~\ref{thm:main2} \\% & \\
\bottomrule
\end{tabular}
\caption{Summary of results on nondegenerate multistationarity for small reactions.  Here $r$~denotes the number of reactions and $s$ the number of species.  See Section~\ref{sec:bkrd}.}
\label{tab:results}
\end{table}

The reader may be wondering what we gain in focusing on small networks, rather than larger networks coming from applications.  The reason stems from a number of recent results on how a given network's capacity for multistationarity arises from that of certain smaller networks~\cite{BP-inher,Joshi:Shiu:Atoms}.  
Here is one such ``lifting'' result, stated informally:
if $N$ is a subnetwork of $G$ 
and both networks have the same number of conservation laws, then if $N$ is nondegenerately multistationary, then $G$ is too (see Lemma~\ref{lem:lift}).  
Therefore, we would like a catalogue of small nondegenerately multistationary networks against which the networks $N$ can be checked.  Our work is therefore a step in this direction, following earlier 
work~\cite{BP-inher,FSW,Joshi:Shiu:Atoms,Joshi:Shiu:Multistationary}. %mss-review

The techniques we harness in this work are largely algebraic.  Specifically, we %establish our main results by proving
prove Proposition~\ref{Prop:PolyTwoDistincPositiveRoots}, which concerns the following univariate polynomial:
 \begin{eqnarray*} 
	g(z) & = & (T-\mu z)^{n_2} - l z^{p_1} (T- \mu z)^{n_1} + m z^{p_2}~,  \end{eqnarray*}
	where 
$\mu>0$ and $1 \leq p_1 < p_2$ and $0 \leq n_1 < n_2$.
We show that if there exist parameters $(T,l,m) \in \mathbb{R}^3_{>0}$ such that 
the polynomial admits two or more positive real roots, 
then we can perturb the parameters so that the polynomial admits two or more {\em multiplicity-one} roots.

While such a result is straightforward for a univariate polynomial with {\em arbitrary} coefficients, here the coefficients of $g(z)$ depend only on $T,l$, and $m$ although the degree of $g(z)$ is arbitrarily high.  Thus, the coefficients satisfy relations which might {\em a priori} preclude simple real roots. Indeed, such obstructions and other similar obstructions occur for sparse polynomials; for instance, trinomials with 
 coprime exponents admit at most three distinct real roots (see \cite[Theorem 4.8 and the following remark]{Theobald:deWolff:Trinomials}).

Accordingly, like \cite{Dickenstein:Invitation}, this work is an invitation to real algebraic geometers.  

We hope to convey that the study of reaction systems leads to interesting problems in real algebraic geometry. Indeed, algebraic techniques, such as elimination of variables and steady-state parametrizations, have already contributed significantly to recent progress in the field, e.g., \cite{CFMW,perspective,dexter2015,Karin02,case-study,messi,sweeney}.%TDS, DickenPerez, Domijan2009, QSSA

The outline of our work is as follows. 
Section~\ref{sec:bkrd} provides background on chemical reaction systems --
including a summary of prior results on nondegenerate multistationarity for small networks -- and configurations of polynomials.  We state our main results in Section~\ref{sec:results} and then prove them in Section~\ref{sec:proof}.  
In Section~\ref{sec:3+}, we describe our efforts toward extending our results to more species.  
Finally, we end with a Discussion in Section~\ref{sec:disc}.

% BACKGROUND
\section{Background} \label{sec:bkrd}
In this section we provide background on chemical reaction systems (Section~\ref{sec:CRS}),
their steady states (Section~\ref{sec:steady-state}),
and polynomials and their discriminants (Section~\ref{sec:discriminant}).

\subsection{Chemical reaction systems} \label{sec:CRS}
Our introduction to chemical reaction systems follows closely the notation in~\cite{Joshi:Shiu:Multistationary}. 

An example of a {{\em \struc{chemical reaction}}} is $A+B \to 3A + C$, 
in which one unit of chemical {\em \struc{species}} $A$ and one of $B$ react 
to form three units of $A$ and one of $C$.  
The  {\em \struc{reactant}} $A+B$ and the {\em \struc{product}} $3A+C$ are called {\em \struc{complexes}}. 
 A reaction network
consists of finitely many reactions (see Definition~\ref{def:crn}).

\begin{definition} \label{def:crn}
A \struc{{\em reaction network}} $\struc{G}:=(\SS,\CC,\RR)$
consists of three finite sets:
\begin{enumerate}
\item a set of \struc{{\em species}} $\struc{\SS} := \{A_1,A_2,\dots, A_s\}$, 
\item a set  $\struc{\CC} := \{y_1, y_2, \dots, y_p\}$ of \struc{{\em complexes}} (finite nonnegative-integer combinations of the species), and 
\item a set of \struc{{\em reactions}}, which are ordered pairs of complexes, excluding diagonal pairs: $\struc{\RR} \subseteq (\CC \times \CC) \setminus \{ (y,y) \mid y \in \CC\}$.
\end{enumerate}
A \struc{{\em subnetwork}} of a network $G = (\SS, \CC, \RR)$ is a network $\struc{G'} := (\SS', \CC', \RR')$ with 
$\SS' \subseteq \SS$, 
$\CC' \subseteq \CC$, and 
$\RR' \subseteq \RR$.
\end{definition}

\noindent
Throughout this work, $s$ and $r$ denote the numbers of species and reactions, respectively.  
A reaction $y_i \to y_j$ is \struc{{\em reversible}} if its {\em reverse reaction} $y_j \to y_i$ is also in $\RR$; 
we denote such a pair by $y_i \rightleftharpoons y_j$.

We write the $i$-th complex as $y_{i1} A_1 + y_{i2} A_2 + \cdots + y_{is}A_s$ (where $y_{ij} \in \mathbb{Z}_{\geq 0}$ 
is the \struc{{\em stoichiometric coefficient}} of $A_j$, for $j=1,2,\dots,s$), which defines the following monomial:
$$ \struc{\mathbf{x}^{y_i}} \,\,\, := \,\,\, x_1^{y_{i1}} x_2^{y_{i2}} \cdots  x_s^{y_{is}}~. $$
For example, the two complexes in the reaction $A+B \to 3A + C$ yield 
the monomials $x_{A}x_{B}$ and $x^3_A x_C$, which determine the vectors 
$y_1=(1,1,0)$ and $y_2=(3,0,1)$.  
These vectors form the rows of a $p \times s$-matrix of nonnegative integers,
denoted by $\struc{Y}:=(y_{ij})$.
Next, the unknowns $\struc{x_1},\struc{x_2},\ldots,\struc{x_s}$ denote the
concentrations of the $s$ species in the network,
and we view them as functions $\struc{x_i(t)}$ of time $t$.

For a reaction $y_i \to y_j$ from the $i$-th complex to the $j$-th
complex, the \struc{{\em reaction vector}}
 $\struc{y_j-y_i}$ encodes the
net change in each species that results when the reaction takes
place.  The \struc{{\em stoichiometric matrix}} $\struc{\Gamma}$ is the $s \times r$ matrix whose $k$-th column 
is the reaction vector of the $k$-th reaction, that is, it is the vector $y_j - y_i$ if $k$ indexes the reaction $y_i \to y_j$.
We associate to each reaction a \struc{{\em rate constant}} $\struc{\kappa_{ij}}$, which is a positive parameter.  

%---------------------------------
% The ODEs
%---------------------------------

% R
The choice of kinetics is represented by a locally Lipschitz function $R:\Rnn^s \to \R^r$ that encodes the reaction rates of the $r$ reactions as functions of the $s$ species concentrations. The \struc{{\em reaction kinetics system}} defined by a reaction network $G$ and reaction rate function $R$ is given by the following system of ODEs:
\begin{align} \label{eq:ODE}
\frac{d \mathbf{x}}{dt} ~ = ~ \Gamma \cdot R( \mathbf{x})~.
\end{align}
% MASS-ACTION
For \struc{{\em mass-action kinetics}}, the assumption for this work, the coordinates of $R$ are
$ R_k( \mathbf{x})=  \kappa_{ij}  \mathbf{x}^{y_i}$, 
 if $k$ indexes the reaction $y_i \to y_j$.  
A \struc{{\em chemical reaction system}} refers to the 
dynamical system (\ref{eq:ODE}) arising from a chemical reaction
network $(\SS, \CC, \RR)$ and a choice of rate constants $(\kappa^*_{ij}) \in
\mathbb{R}^{r}_{>0}$ (recall that $r$ is the number of
reactions) where the reaction rate function $R$ is that of mass-action
kinetics.  Specifically, the mass-action ODEs are the following ones:
% ODE's of MASS-ACTION
\begin{align} \label{eq:ODE-mass-action}
\frac{d\mathbf{x}}{dt} \quad = \quad \sum_{ y_i \to y_j~ {\rm is~in~} \RR} \kappa_{ij} \mathbf{x}^{y_i}(y_j - y_i) \quad =: \quad \struc{f_{\kappa}(\mathbf{x})}~.
\end{align}

The \struc{{\em stoichiometric subspace}}, denoted by $S$,
 is the vector subspace of
$\mathbb{R}^s$ spanned by the reaction vectors
$y_j-y_i$:
\begin{equation*}% \label{eq:stoic_subs}
  \struc{\St}~:=~ {\rm span} \left( \{ y_j-y_i \mid  y_i \to y_j~ {\rm is~in~} \RR \} \right)~.
\end{equation*}
Note that $\St = \im(\Gamma)$, where $\Gamma$ is the stoichiometric matrix.
For the network consisting of the single reaction $A+B \to 3A + C$, we have that $y_2-y_1 =(2,-1,1)$ spans $\St$.

The  vector $\frac{d \mathbf{x}}{dt}$ in  (\ref{eq:ODE}) lies in
$\St$ for all time $t$.   
In fact, a trajectory $\mathbf{x}(t)$ beginning at a positive vector $\mathbf{x}(0)=\mathbf{x}^0 \in
\R^s_{>0}$ remains in the following \struc{{\em stoichiometric compatibility class}}:
\begin{align}\label{eqn:invtPoly}
\struc{\invtPoly}~:=~(\mathbf{x}^0+\St) \cap \mathbb{R}^s_{\geq 0}~
\end{align}
for all positive time.  That is, $\invtPoly$ is forward-invariant with
respect to the dynamics~(\ref{eq:ODE}).    

\begin{example} \label{ex:main}
Consider again the network from the introduction:
	\[
	\left\{
		0 \overset{k_1}{\underset{k_2}\lra} A+B \quad \quad 
		2A+B 
		\overset{k_3}\to 3A+2B
	\right\}~.
	\]
The mass-action ODEs are:
	\[
	\frac{dx_A}{dt}~=~
	\frac{dx_B}{dt}~=~
	k_1 - k_2 x_A x_B + k_3 x_A^2 x_B~,
	\]
and the stoichiometric subspace is $S={\rm span}\{(1,1)^t\}$. Thus, the stoichiometric compatibility classes are the rays 
$\invtPoly = \{ (a,a+T) \mid a \geq 0,~ a+T \geq 0 \}$ (for some $T \in \mathbb{R}$) in Figure~\ref{fig:classes}.
\end{example}

\begin{center}
\begin{figure}[ht]
\begin{tikzpicture}[scale=0.7]
\begin{axis}[xmin=0,xmax=2,ymin=0,ymax=2, axis lines = middle, xlabel=$x_A$,ylabel=$x_B$]
\addplot +[mark=none,color=blue,dashed] coordinates {(0,0) (2.5,2.5)};
\addplot +[mark=none,color=blue,dashed] coordinates {(0.5,0) (2.5,2.0)};
\addplot +[mark=none,color=blue,dashed] coordinates {(1,0) (2.5,1.5)};
\addplot +[mark=none,color=blue,dashed] coordinates {(1.5,0) (2.5,1)};
\addplot +[mark=none,color=blue,dashed] coordinates {(0,0.5) (2,2.5)};
\addplot +[mark=none,color=blue,dashed] coordinates {(0,1) (1.5,2.5)};
\addplot +[mark=none,color=blue,dashed] coordinates {(0,1.5) (1,2.5)};
\end{axis}
\end{tikzpicture}
\caption{Stoichiometric compatibility classes for the network in Example~\ref{ex:main}.
\label{fig:classes}}
\end{figure}
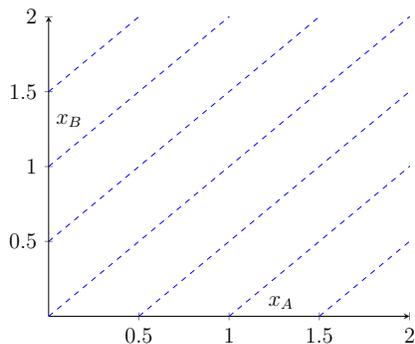
\end{center}

\subsection{Steady states} \label{sec:steady-state}
% DEFINITIONS: steady state, non degenerate, etc.
A \struc{{\em steady state}} of a reaction kinetics system is a nonnegative concentration vector $\struc{\mathbf{x}^*} \in \Rnn^s$ at which the right-hand side of the ODEs~\eqref{eq:ODE}  vanish: $f_{\kappa} (\mathbf{x}^*) = 0$.  
A steady state $\mathbf{x}^*$ is \struc{{\em nondegenerate}} if ${\rm Im}\left( df_{\kappa} (\mathbf{x}^*)|_{S} \right) = \St$, where $\struc{df_{\kappa}(\mathbf{x}^*)}$ is the Jacobian matrix of $f_{\kappa}$ at $\mathbf{x}^*$.
We are interested in \struc{{\em positive steady states}} $\mathbf{x} ^* \in \mathbb{R}^s_{> 0}$.

\medskip

% DEFINITION: MSS, admits, capacity!  
\begin{definition}~ \label{def:mss}
\begin{enumerate}
% DEF: multistationary and NONDEGENERATELY MULTISTATIONARY
\item  A reaction kinetics system~\eqref{eq:ODE} is \struc{{\em multistationary}} 
if there exists a stoichiometric compatibility class~\eqref{eqn:invtPoly}
with two or more positive steady states. 
Similarly, a reaction kinetics system is \struc{{\em nondegenerately multistationary}}
 if it admits two or more nondegenerate
 positive steady states in some stoichiometric compatibility class.

\item A network is \struc{{\em multistationary}}
if there exist some choice of 
positive rate constants $\kappa_{ij}$
such that the resulting 
mass-action kinetics system~\eqref{eq:ODE-mass-action}
is multistationary.  
Analogously, a network may be 
\struc{{\em nondegenerately multistationary}}.

\item % ADMITS
A network \struc{{\em admits $k$ positive steady states}} if there exists a choice of positive rate constants so that the resulting mass-action system
has exactly $k$ positive steady states in some stoichiometric compatibility class.
Similarly, a network may \struc{{\em admit $k$ nondegenerate positive steady states}}.
\end{enumerate}

\end{definition}

We now state the conjecture mentioned in the introduction.

\begin{conjecture}[Nondegeneracy Conjecture \cite{Joshi:Shiu:Multistationary}] \label{conj:nondeg}
Consider a network $G$ that does not admit infinitely many positive steady states (in any stoichiometric compatibility class).  Then if $G$ admits $k$ positive steady states, then $G$ admits $k$ {\em nondegenerate} positive steady states.
\end{conjecture}

We know only two classes of network for which this conjecture has been proven: networks with only one species~\cite[Theorem 3.6]{Joshi:Shiu:Atoms}, and networks with up to two reactions (see \cite[Theorems 5.1 and 5.2]{Joshi:Shiu:Atoms} and their proofs).  One goal of the present work is to resolve the conjecture for 2-species networks comprising one irreversible reaction and one reversible-reaction pair, or two reversible-reaction pairs.

%Next we recall 
We will use two results of Joshi and Shiu.
% that we use.  
The first% result
~\cite[Theorem~3.1]{Joshi:Shiu:Atoms} ``lifts''  steady states from a subnetwork to a larger network if they share the same stoichiometric subspace:
%---------------------------------
% LEMMA -- LIFTING RESULT
%----------------------------------
\begin{lemma} \label{lem:lift}
Let $N$ be a subnetwork of a reaction network $G$ that has the same stoichiometric subspace as $G$.  
If $N$ admits $m$ nondegenerate 
positive steady states (in some stoichiometric compatibility class, for some choice of rate constants), then $G$ admits at least $m$ nondegenerate positive steady states.
\end{lemma} 

To state the second result, we must recall some definitions from~\cite{Joshi:Shiu:Multistationary}.

%----------------------------------
% DEF: Arrow diagram
%----------------------------------
\begin{definition} \label{def:arrow-diagram}
Let $G$ be a reaction network that contains only one species $A$. So, each reaction of $G$ has the form $aA \to bA$, where $a,b \ge 0$ and $a \ne b$. Let $m$ be the number of (distinct) reactant complexes, and let $a_1< a_2 < \ldots < a_m$ be their stoichiometric coefficients. The 
\struc{{\em arrow diagram}} of $G$, denoted $\struc{\rho} = (\rho_1, \ldots , \rho_m)$, is the element of $\{\to , \la, \lradot \}^m$ where:
\begin{equation*}
 \struc{\rho_i}~:=~ 
 \left\lbrace\begin{array}{ll}
   \to & \text{if for all reactions $a_iA \to bA$ in $G$, it is the case that $b > a_i$} \\
   \la & \text{if for all reactions $a_iA \to bA$ in $G$, it is the case that $b < a_i$} \\
   \lradot & \text{otherwise.}
 \end{array}\right.
\end{equation*}
\end{definition}

%----------------------------------
% DEF: T-alternating
%----------------------------------
\begin{definition} \label{def:2-sign-change}
For positive integers $T \geq 1$, a \struc{{\em $T$-alternating network}} is a 1-species network with 
exactly $T+1$ reactions and with arrow diagram $\rho \in \{\to , \la\}^{T+1}$ such that $\rho_i = \to$ if and only if $\rho_{i+1} = \la$ for all $i \in \{1, \ldots, T\}$. 
\end{definition}

%----------------------------------
% EX: alternating networks
%----------------------------------
\begin{example} \label{ex:alternating} 
Consider the following network:
 	\begin{align*}
	G~=~\{ 0 \leftarrow A \to 2A \lra 3A\}~.
	\end{align*}
 Two 1-alternating subnetworks of $G$ 
	have arrow diagram $(\to, \la)$:
	$\{ A \to 2A,~ 2A \leftarrow 3A\}$ and
	$\{  2A \to 3A,~ 2A \leftarrow 3A\}$.
On the other hand, $\{ 0 \leftarrow A,~ A \to 2A\}$ is {\em not} a 1-alternating subnetwork of $G$: %(the two reactants are equal)  
	its arrow diagram is $(\lradot)$. % of the form ``$\leftarrow \to$''.
Finally, $\{ 0 \leftarrow A,~ 2A \to 3A, ~ 2A \leftarrow 3A \}$ is a 2-alternating subnetwork of $G$ with arrow diagram $(\la, \to, \la)$.
\end{example}

Next we define reactant polytopes (Newton polytopes) and box diagrams.

%----------------------------------
% DEF: reactant polytope
%----------------------------------
\begin{definition}[\cite{GMS2}] \label{def:NP}
The \struc{{\em reactant polytope}} of a network $G$ is the convex hull of 
%(i.e., the smallest convex set containing) 
the reactants of $G$ (in $\mathbb{R}^s$, where $s$ is the number of species),
that is, the smallest convex set containing the set
$  \{ y_i \mid  y_i \to y_j~ {\rm is~in~} \RR ~ ({\rm for~some~}j) \} $.
\end{definition}

%----------------------------------
% DEF: box diagram
%----------------------------------
\begin{definition} \label{def:box}
Let $G$ be a network with exactly two species and two reactions, $y \to y'$ and $\widetilde y \to \widetilde y'$, such that the reactant vectors differ in both coordinates (i.e., writing $y=(y_A, y_B)$ and $\widetilde y=(\widetilde y_A, \widetilde y_B)$, then both $y_A \neq \widetilde y_A$ and  $y_B \neq \widetilde y_B$).  The \struc{{\em \bd}} of the network $G$ is the rectangle in $\mathbb{R}^2$ such that
	\begin{enumerate}
	\item the edges are parallel to the axes of $\mathbb{R}^2$, and
	\item the reactants $y$ and $y'$ are two opposite corners of the rectangle.
	\end{enumerate}	
\end{definition}

\begin{remark}
The box diagram is the smallest rectangle containing the reactant polytope.
\end{remark}

We depict a box diagram together with the reaction vectors and the reactant polytope (which in this case is the diagonal of the box that connects the two reactants).  For example, 
consider the network $\{ A \to B , ~ 2A+B \to 3A \}$, which is equivalent to a network considered 
in~\cite[\S 6]{FeinDefZeroOne} and also equivalent to a subnetwork of a bistable network modeling apoptosis~\cite{HoHarrington}.  The box diagram is:
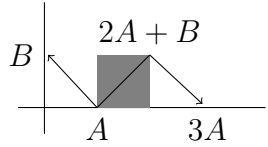
\begin{center}
	\begin{tikzpicture}[scale=.7]
   % axes
	\draw (-0.5,0) -- (4.2,0);
	\draw (0,-.5) -- (0,2);
  % reactions
	\draw [->] (1,0) -- (0.07,1);
	\draw [->] (2,1) -- (3,0.07);
  % box
	\path [fill=gray] (1,0) rectangle (2,1);
  % diagonal (NP)
	\draw (1,0) --(2,1); 
    % labels
    \node [below] at (1,0) {$A$};
    \node [left] at (0,1) {$B$};
    \node [above] at (2,1) {$2A+B$};
    \node [below] at (3,0) {~$3A$};
	\end{tikzpicture}
\end{center}
This box diagram has the form of one of the four depicted in Proposition~\ref{prop:main-prior}, part (3) below, so we conclude, by inspection, that the network is nondegenerately multistationary.  Indeed, one of our goals is to obtain more easy-to-check criteria for nondegenerate multistationarity.

The following result is~\cite[Theorems 3.6 and 5.2]{Joshi:Shiu:Multistationary} (and summarized
in Table~\ref{tab:results}):
%----------------------------------
% THEOREM: stated in introduction
%----------------------------------
\begin{proposition} \label{prop:main-prior}
Let $G$ be a reaction network with exactly $r$ reactions and $s$ species. Then:
\begin{enumerate}
	\item If $s=1$, then $G$ is nondegenerately multistationary if and only if $G$ has a 2-alternating subnetwork (i.e., with arrow diagram $(\to, \la, \to)$ or $(\la, \to, \la)$). 
	\item If $r=1$ or $G$ consists of a reversible-reaction pair, then $G$ is not multistationary. 
	\item If $r=2$, then $G$ is nondegenerately multistationary if and only if for some choice of species $i$ and $j$, the projection of the box diagram to the $(i,j)$-plane has one of the following ``zigzag'' forms:
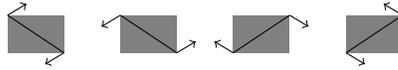
\begin{center}
	\begin{tikzpicture}[scale=.5]
  %-------
  % box 1
  %-------
	\path [fill=gray] (0,0) rectangle (1.5,1);
  % reactions
	\draw [->] (0,1) -- (.5,1.3);
	\draw [->] (1.5,0) -- (1,-.3);
  % diagonal (NP)
	\draw (0,1) --(1.5,0); 
  %-------
  % box 2
  %-------
	\path [fill=gray] (3,0) rectangle (4.5,1);
  % reactions
	\draw [->] (3,1) -- (2.5,.7);
	\draw [->] (4.5,0) -- (5,.3);
  % diagonal (NP)
	\draw (3,1) --(4.5,0); 
  %-------
  % box 3
  %-------
	\path [fill=gray] (6,0) rectangle (7.5,1);
  % reactions
	\draw [->] (6,0) -- (5.5,.3);
	\draw [->] (7.5,1) -- (8, .7);
  % diagonal (NP)
	\draw (6,0) --(7.5,1); 
  %-------
  % box 4
  %-------
	\path [fill=gray] (9,0) rectangle (10.5,1);
  % reactions
	\draw [->] (9,0) -- (9.5, -.3);
	\draw [->] (10.5,1) -- (10,1.3);
  % diagonal (NP)
	\draw (9,0) --(10.5,1);
	\end{tikzpicture}
\end{center}
and, if only one such pair $(i,j)$ exists, then the slope of the marked diagonal is {\em not} $-1$.
\end{enumerate}
\end{proposition}
Proposition~\ref{prop:main-prior} says that the classification of nondegenerately multistationary networks is already complete for networks with 1 species or 1 or 2 reactions.
Thus, in this work we tackle the next cases, those of 1 irreversible reaction and 1 reversible-reaction pair, or 2 reversible-reaction pairs -- 
under the assumption of only 2 species (Theorems~\ref{thm:main} and~\ref{thm:main2}).
These results, prior and new, on nondegenerate multistationarity for small networks are summarized
Table~\ref{tab:results}.

\subsection{Discriminants and configuration spaces of polynomials} \label{sec:discriminant}

Let $d \in \N$. We consider the \struc{\textit{configuration space}} of univariate  polynomials of degree at most $d$:
\begin{eqnarray*}
\struc{\C_d[z]} & := & \left\{ b_d z^d + b_{d-1} z^{d-1} + \dots + b_0 \mid b_0, b_1,\ldots,b_d \in \C\right\}.
\end{eqnarray*}
Every polynomial $f \in \C_d[z]$ is uniquely determined by its coefficient vector, so $\C_d[z]$ is isomorphic (as a vector space over $\mathbb{C}$) to $\C^{d+1}$.

Note that $\C_d[z]$ is a metric space induced by the Euclidean norm of the difference of the corresponding coefficient vectors. 
We denote this metric by \struc{$\dist(\cdot,\cdot)$}. For every $f \in \C_d[z]$, we define 
\begin{align*}
\struc{\cV(f)} ~:=~ \{v \in \C \mid f(v) = 0\}~.
\end{align*}

It is well-known that roots of univariate polynomials $f$ are continuous with respect to the coefficients of $f$; see e.g.\ \cite[Theorem 1.3.1, page 10]{Rahman:Schmeisser:AnalyticTheoryOfPolynomials}.

\begin{theorem}
The function $\struc{\cV} : \C_d[z] \to \Sym_d(\C)$, %^{d+1}
 given by $f \mapsto \cV(f)$, is continuous.
\label{Thm:RootsAreContinuousInCoefficients}
\end{theorem}

Following Gelfand, Kapranov, Zelevinsky \cite[Chapter 9]{Gelfand:Kapranov:Zelevinsky} we define the subset
\begin{eqnarray*}
\struc{\nabla_0} & := & \{f \in \C_d[z] \mid \text{ there exists } v \in \C \setminus \{0\} \text{ with } f(v) = f'(v) = 0\}~,
\end{eqnarray*}
and let $\struc{\nabla}$ denote its Zariski closure.
It is well known that $\struc{\nabla}$ is a hypersurface defined by a single polynomial~\cite{Gelfand:Kapranov:Zelevinsky}.

\begin{theorem}
For $d \geq 2$ the set $\nabla$ is a hypersurface in $\C_d[z]$, and there exists an irreducible, integral polynomial $\struc{\Delta} \in \mathbb{Z}[b_0,\ldots,b_d]$ such that $\cV(\Delta) = \nabla$, which is unique up to sign.
\label{Thm:DiscriminantHypersurface}
\end{theorem}

The polynomial $\Delta$
is the \struc{\textit{discriminant}} for $\C_d[z]$. The set $\C_d[z] \setminus \nabla$ is a well-studied mathematical object with various applications, e.g., for knot theory or Morse functions~\cite{Vassiliev:Discriminants}. 

% SECTION- MAIN RESULT
\section{Main results} \label{sec:results}
Our main results (Theorems~\ref{thm:main} and~\ref{thm:main2}) strengthen, in the case of 2 species, the classification of multistationary networks with one reversible-reaction pair and one irreversible reaction, or two reversible-reaction pairs (Theorem~\ref{prop:i=1-r=1}). 
Our results state that these multistationary networks are indeed nondegenerately multistationary, thereby lending support for the Nondegeneracy Conjecture (Conjecture~\ref{conj:nondeg}).

To state Theorem~\ref{thm:main}, we must introduce ``embedded'' networks, which generalize subnetworks.  A subnetwork $N$ is obtained from a reaction network $G$ by removing a subset of reactions (that is, setting some of the reaction rates to 0), while an embedded network is obtained by removing a subset of reactions and/or species.
For instance, removing the species $B$ from the reaction $A+B \to A+C$ yields the reaction $A \to A+C$.

%DEFINITION OF RESTRICTION
\begin{definition}
The {\em \struc{restriction}} of a set of reactions $\RR$ to a set of species $\SS$, denoted by $\struc{\RR |_{\SS}}$, is the subset of $\RR$ remaining after (1) setting to 0 the stoichiometric coefficients of all species
not in $\SS$, and then (2) discarding any {\em \struc{trivial reactions}} (reactions of the form $\sum m_i A_i \to \sum m_i A_i$, i.e., when the source complex equals the product) and keeping only one copy of any duplicate reactions.
\end{definition}

\begin{definition} \label{def:emb}
The {\em \struc{embedded network}} $\struc{N}$ of a network $G := (\SS,\CC,\RR)$ 
 obtained by removing a set of reactions $\{ y \to y' \} \subseteq \RR$ and 
 a set of species  $\{X_i\} \subseteq \SS$ is
\[
\struc{N} \ := \ \left( \SS|_{\CC|_{\RR_N}},~ \CC|_{\RR_N}, \ \RR_N := \left(\RR \setminus  \{ y \to y' \}\right)|_{\SS \setminus \{X_i\}}\right)~,
\]
where
$ \CC|_{\RR_N}$ denotes the set of complexes of the set of reactions $\RR_N$, and 
$\SS|_{\CC|_{\RR_N}}$ denotes the set of species in the set of complexes 
$\CC|_{\RR_N}$.
\end{definition}

\begin{example} \label{ex:net-no-MSS}
Consider the network $G = \{2B \lra A+B~, ~ 2A+B \leftarrow 3A\}$.  
Its 1-species embedded networks are 
$\{0 \lra A~, ~ 2A \leftarrow 3A \}$ and 
$\{0 \to B \lra 2B\}$,
neither of which is 2-alternating.  Hence, by Theorem~\ref{prop:i=1-r=1},
network $G$ is not multistationary (and thus not nondegenerately multistationary).
\end{example}

\begin{example} \label{ex:net-MSS}
Recall the network $G = \{0 \lra A+B~, ~ 2A+B \to 3A+2B\}$
from Example~\ref{ex:main}.
The 1-species embedded network 
$\{0 \lra A~, ~ 2A \to 3A \}$ 
is 2-alternating (informally, it has the form ``$\lra \, \, \to$'').
Also, the reaction vectors are scalar multiples of each other.  
So, by Theorem~\ref{prop:i=1-r=1},
network $G$ is multistationary.
In fact, we see next that $G$ is nondegenerately multistationary (Theorem~\ref{thm:main}).  
No prior work 
yields this result (see Table~\ref{tab:results}).
\end{example}

%----------------------------------
% MAIN THEOREM -- i=1, r=1, AND s=2
%----------------------------------
\begin{theorem}[Classification of nondegenerately multistationary, 2-species networks with one reversible reaction and one irreversible reaction] \label{thm:main}
Let $G$ be a 2-species network that consists of one reversible-reaction pair $y \lra y'$
and one irreversible reaction $\widetilde y \to \widetilde y'$.  
Then the following statements are equivalent:
\begin{enumerate}[(1)]
\item $G$ is nondegenerately multistationary. 
\item the reaction vectors are (nontrivial) scalar multiples of each other: 
 $y'-y = \lambda (\widetilde y' - \widetilde y)$ for some $0 \neq \lambda \in \mathbb{R}$, and, for some species $i$, the embedded network of $G$ obtained by removing all species except $i$ is a \underline{2-alternating} network (``$\lra \, \, \to$''
or
 ``$ \leftarrow \, \, \lra$'').%
\end{enumerate}
\end{theorem}

Theorem~\ref{thm:main}, which we prove in Section~\ref{sec:proof}, yields the following result:

%----------------------------------
% MAIN THEOREM 2 --  rev=2 AND s=2
%----------------------------------
\begin{theorem}[Classification of nondegenerately multistationary, 2-species networks with two reversible-reaction pairs] \label{thm:main2}
Let $G$ be a 2-species network 
that consists of two reversible-reaction pairs, $y \lra y'$ and $\widetilde y \lra \widetilde y'$.  
Then the following statements are equivalent:
\begin{enumerate}[(1)]
\item $G$ is nondegenerately multistationary. 
\item the reaction vectors are (nontrivial) scalar multiples of each other: 
 $y'-y =  \lambda (\widetilde y' - \widetilde y)$ for some $0 \neq \lambda \in \mathbb{R}$, and, 
 for some species $i$, the embedded network of $G$ obtained by removing all species except $i$ is a \underline{3-alternating} network (``$\lra \, \, \lra$'').%
\end{enumerate}
\end{theorem}

% PROOF
\begin{proof}
First, (1) $\Rightarrow$ (2) follows immediately from Theorem~\ref{prop:i=1-r=1}.  As for the converse,
(2) says that $G$ has a one-dimensional stoichiometric subspace and has a 1-species embedded network that is 3-alternating  (``$\lra \, \, \lra$''), 
which therefore has a 2-alternating subnetwork of the form ``$\lra \, \, \to$'' (and in fact also has one of the form  ``$ \leftarrow \, \, \lra$'').  
Thus, by Theorem~\ref{thm:main}, the corresponding subnetwork $N$ of $G$ is nondegenerately multistationary.  So, by Lemma~\ref{lem:lift}, $G$ too is nondegenerately multistationary.
\end{proof}

% SECTION: PROOF
\section{Proof of the main result} \label{sec:proof}
The main technical piece for proving Theorem~\ref{thm:main} is the following proposition:
% PROPOSITION
\begin{proposition} \label{Prop:PolyTwoDistincPositiveRoots}
Fix $\mu>0$ and integers $p_1,p_2,n_1,n_2$ for which $1 \leq p_1 < p_2$ and $0 \leq n_1 < n_2$, and 
consider the following polynomial:
 \begin{eqnarray} 
	g(z) & := & (T-\mu z)^{n_2} - l z^{p_1} (T- \mu z)^{n_1} + m z^{p_2}~. \label{Equ:OurPolynomialDefinition}
 \end{eqnarray}
Assume that there exists $(T,l,m) \in \mathbb{R}^3_{>0}$ for which $g(z)$ admits two or more {\em distinct} real roots
in the interval $(0, T/\mu)$.  
Then there exists $(\widetilde T,\widetilde l,\widetilde m) \in \mathbb{R}^3_{>0}$ yielding a polynomial $\widetilde{g}$ of the form \eqref{Equ:OurPolynomialDefinition} that admits two or more (distinct) {\em multiplicity-one} roots in $(0, \widetilde T/\mu)$.  
\end{proposition}

Before proving Proposition~\ref{Prop:PolyTwoDistincPositiveRoots}, we recall why it is nontrivial. 
Given Theorem~\ref{Thm:DiscriminantHypersurface}, the proposition would be trivial if we were instead considering a general polynomial in $\C_d[z]$ (where 
$d := \max\{n_2,~n_1 + p_1,~p_2\}$, 
as the zero set $\nabla$ of the discriminant is codimension-one in this space). However, we are considering only a 
three-dimensional subset of $\R_{d}[z]:=\left\{ b_d z^d + b_{d-1} z^{d-1} + \dots + b_0 \mid b_0, b_1,\ldots,b_d \in \R \right\}$, 
arising from~\eqref{Equ:OurPolynomialDefinition},
which {\em a priori} could be contained in $\nabla$. 

As a first step towards a proof of Proposition \ref{Prop:PolyTwoDistincPositiveRoots} we show the following lemma.

% LEMMA
\begin{lemma}
Let the notation be as in Proposition \ref{Prop:PolyTwoDistincPositiveRoots}. Assume that $b$
 is a root of $g$ in the interval $(0, T/\mu)$. Then for all 
$\epsi > 0$, there exists a polynomial $\widetilde{g}$ of the form \eqref{Equ:OurPolynomialDefinition} with parameters $(\widetilde{T},\widetilde{l},\widetilde{m}) \in \mathbb{R}^3_{>0}$ such that 
\begin{itemize}
\item $\dist(g,\widetilde g) < \epsi$ and 
\item $\widetilde{g}(b) = 0$, 
and $b$ has multiplicity one.
\end{itemize}
\label{Lem:ReduceMultiplicityOfRoots}
\end{lemma}

% PROOF
\begin{proof}
Let $g$ be as in Proposition \ref{Prop:PolyTwoDistincPositiveRoots}, with parameters $(T,m,l)$. First, we claim that we can assume $\mu = 1$. Indeed,
if $1 \neq \mu \in \R_{> 0}$, then consider the isomorphism $\R \to \R$, 
given by $z \mapsto z / \mu$, and replace $l$ by $l \cdot \mu^{p_1}$ and $m$ by $m \cdot \mu^{p_2}$. We can carry out this replacement, because we are only interested in the multiplicity of roots, and thereby obtain an equivalent $\mu=1$ version of $g$.

We rearrange $g(z)$ as follows:
\begin{equation} \label{eq:rearrange}
 {g}(z) \ = \ (T- z)^{n_1}
	\left[(T - z)^{n_2-n_1} - (l+z^{p_2}) z^{p_1} \right] 
	+ (m + (T- z)^{n_1} z^{p_1}) z^{p_2}~.
\end{equation}
Assume that $b \in (0,T/\mu) = (0,T)$  is a root of $g$. It is straightforward to check from~\eqref{eq:rearrange} that
there exists a one-dimensional subspace of polynomials of the form~\eqref{Equ:OurPolynomialDefinition}
with the same root $b$; namely, these polynomials are defined by the parameters
$(T, \widetilde{l},\widetilde{m})$, where:
\begin{eqnarray}\label{Equ:LinearInvariantSubspace}
 \widetilde{l} ~:=~ l+\lam b^{p_2} 
 \quad \quad {\rm and} \quad \quad
 \widetilde{m} ~:=~m + \lam (T- b)^{n_1} b^{p_1}~,
\end{eqnarray}
for any choice of  $\lam \in \R$.

Fix $\epsi >0$.  
To complete the proof, it suffices to show that there exists 
$\lam > 0$ such that for the polynomial $\widetilde{g}(z)$ given by the induced parameters $(T, \widetilde{l},\widetilde{m})$,
as in~\eqref{Equ:LinearInvariantSubspace},
 it holds that $\dist(g,\tilde{g}) < \epsi$ and $b$ is a root of multiplicity one for $\widetilde{g}$. Hence, for the rest of the proof, we assume (for contradiction) that no such $\lam$ exists.

In particular, 
for $\lambda$ sufficiently small, 
$b$ is a multiple root of $g$ in~\eqref{Equ:OurPolynomialDefinition},
where $(T, \widetilde{l},\widetilde{m})$ are as in~\eqref{Equ:LinearInvariantSubspace}.
Thus, $g \in \nabla$ and $g(b) = g'(b)$. We compute, using~\eqref{Equ:OurPolynomialDefinition}:
\begin{eqnarray*}
 g'(b) & = & 
 	-n_2  (T- b)^{n_2 - 1} - \left(p_1 l b^{p_1-1} (T- b)^{n_1} - l b^{p_1} n_1 (T- b)^{n_1 - 1}\right) + p_2 m b^{p_2 -1} \\
 & = &  (T- b)^{n_1 - 1}
 	\left[-n_2 (T- b)^{n_2 - n_1} - l\left(p_1 b^{p_1-1} (T- b) - b^{p_1} n_1  \right) \right] + p_2 m b^{p_2 -1}~.
\end{eqnarray*}
Hence, $g(b) = g'(b)$ is equivalent to:
\begin{eqnarray*}
& & (T- b)^{n_1}\left( (T -  b)^{n_2-n_1} - l b^{p_1}\right) + m b^{p_2} \\
 & = & (T- b)^{n_1 - 1}\left[ -n_2  (T- b)^{n_2 - n_1} - l(p_1 b^{p_1-1} (T- b) - b^{p_1}  n_1)\right] + p_2 m b^{p_2 -1} ~.
\end{eqnarray*}
We rearrange this equation to obtain:
 \begin{eqnarray*}
 & & (T- b)\left((T -  b)^{n_2-n_1} - l b^{p_1}\right) + \frac{m b^{p_2-1}(b - p_2)}{(T- b)^{n_1 - 1}} 
 \\
 & = & -n_2 (T- b)^{n_2 - n_1} - l\left(p_1 b^{p_1-1} (T- b) - b^{p_1}  n_1\right)~,
 \end{eqnarray*}
 and then arrange so that only terms involving $l$ or $m$ appear on the left-hand side:
 \begin{eqnarray} \label{eq:l-m} 
& &  l b^{p_1}\left(\left(\frac{p_1}{b} - 1\right)(T- b) -  n_1 \right) + \frac{m b^{p_2-1}(b - p_2)}{(T- b)^{n_1 - 1}} 
\\ \notag
 & = & 
 	(-n_2  - (T- b))(T- b)^{n_2 - n_1}~.
\end{eqnarray}

Equation~\eqref{eq:l-m} holds equally well when $l$ and $m$ are replaced by, respectively, $\widetilde l$ and $\widetilde m$ as in~\eqref{Equ:LinearInvariantSubspace}, for sufficiently small $\lambda \neq 0$ (because
we have assumed that
 $b$ is a {\em multiple} root of the polynomial~\eqref{Equ:OurPolynomialDefinition} given by $(T, \widetilde{l},\widetilde{m})$).
 Subtracting equation~\eqref{eq:l-m} from the version of equation~\eqref{eq:l-m} obtained by replacing $l$ and $m$ by, respectively, $\widetilde l$ and $\widetilde m$ as in~\eqref{Equ:LinearInvariantSubspace} --- the resulting right-hand side is 0 because the right-hand side of~\eqref{eq:l-m} does not depend on $l$ or $m$ --- we obtain:
\begin{eqnarray}
 \lam b^{p_2} b^{p_1}\left(\left(\frac{p_1}{b} - 1\right)(T- b) -  n_1 \right) & + & \frac{\lam(T- b)^{n_1}b^{p_1} b^{p_2-1}(b - p_2)}{(T- b)^{n_1 - 1}}  \ = \ 0~.
 	 \nonumber %\\
\end{eqnarray}
It is straightforward to simplify this equation (after dividing by $\lambda b^{p_1}b^{p_2}$) to obtain:
\begin{align} \label{eq:difference}
	(p_1 - p_2 )(T- b) ~ = ~ b  n_1 ~.
\end{align}
We have reached a contradiction: 
the left-hand side of equation~\eqref{eq:difference} is negative
 (because $p_1 < p_2$ and $T- b>0$), while
 the right-hand side is non-negative (as $b>0$ and $n_1\geq 0$).
 This contradiction holds for all choices of $\lam \neq 0$, and so completes the proof.
\end{proof}

% PROOF OF PROP
\begin{proof}[Proof of Proposition \ref{Prop:PolyTwoDistincPositiveRoots}]
Consider a polynomial $g$ as given in the proposition. By assumption, $g$ has at least two positive real roots $a_1$ and $a_2$
in the interval $(0,T/\mu)$.
We can assume that at least one root has multiplicity at least two, as otherwise nothing is left to show.
We distinguish several cases.  

\textbf{Case 1:} $a_1$ has multiplicity at least two, and $a_2$ has multiplicity one. Define
\begin{align*}
\delta ~:=~ \min  \left\{ 
	a_1,~a_2,~
	\frac{1}{2} \dist \left(a_1,T/\mu \right),~
	\frac{1}{2} \dist \left(a_2,T/\mu \right),~
	\frac{1}{2} \dist(a_1,a_2)	
	\right\} ~.
\end{align*}
We apply Lemma \ref{Lem:ReduceMultiplicityOfRoots} with respect to $a_1$ and a sufficiently small $\epsi > 0$. We obtain a new polynomial $\widetilde g$ of the form \eqref{Equ:OurPolynomialDefinition} such that $\widetilde g(a_1) = 0$ and $a_1$ has multiplicity one. 
Roots of polynomials are continuous in their coefficients, by Theorem \ref{Thm:RootsAreContinuousInCoefficients}, so we know that every root of $\widetilde{g}$ is in a $\delta$-neighborhood of a root of $g$
(by choosing $\epsi$ sufficiently small).  
Since $\widetilde{g}$ is real, and non-real roots of real polynomials appear in 
complex-conjugate pairs, and $a_2$ is an isolated real root of $g$, there 
must exist an isolated \textit{real} root $\widetilde a_2$ of $\widetilde{g}$ in a $\delta$-neighborhood of $a_2$. 
Finally, we
require that $\epsi < \min \left\{ \frac{1}{2} \dist \left(a_1,T/\mu \right),~\frac{1}{2} \dist \left(a_2,T/\mu \right)\right\}$, so that by construction,
$a_1$ and $\widetilde a_2$ are distinct multiplicity-one roots of $\widetilde g$ 
in the interval $(0,\widetilde T/\mu)$.  

\textbf{Case 2:} Both roots $a_1$ and $a_2$ have multiplicity at least two, 
and one of the roots, say $a_1$, has even multiplicity. We apply Lemma \ref{Lem:ReduceMultiplicityOfRoots} with respect to $a_1$ and a sufficiently small $\epsi > 0$. We obtain a new polynomial $\widetilde g$ of the form \eqref{Equ:OurPolynomialDefinition} such that $\widetilde g(a_1) = 0$ and $a_1$ has multiplicity one. Since $\widetilde g$ is a real polynomial and thus its non-real roots appear in complex-conjugate pairs, $\widetilde{g}$ has another positive, real root $a_3$ in a $\delta$-neighborhood of $a_1$, and $a_3$ has odd multiplicity due to Theorem \ref{Thm:RootsAreContinuousInCoefficients}. If $a_3$ has multiplicity one, then we are done. Otherwise, 
$\widetilde g$ has a root $a_3$ of multiplicity at least two and a root
$a_1$ of multiplicity one, so we are thus reduced to Case 1.

\textbf{Case 3:} Both roots $a_1$ and $a_2$ have odd multiplicity at least three. 
We apply Lemma \ref{Lem:ReduceMultiplicityOfRoots} with respect to $a_1$ and a sufficiently small $\epsi > 0$. We obtain a new polynomial $\widetilde g$ of the form \eqref{Equ:OurPolynomialDefinition} such that $\widetilde g(a_1) = 0$ and $a_1$ has multiplicity one. Since $a_2$ has odd multiplicity and $\widetilde{g}$ is real, $\widetilde{g}$ has a positive real root $\widetilde{a_2}$ in a $\delta$-neighborhood of $a_2$. If $\widetilde{a_2}$ has multiplicity one, then we are done. Otherwise, 
$\widetilde g$ has a root $\widetilde a_2$ of multiplicity at least two and a root
$a_1$ of multiplicity one, so we are again reduced to Case 1.
\end{proof}

We can now prove Theorem~\ref{thm:main}.
\begin{proof}[Proof of Theorem~\ref{thm:main}]
In light of Proposition~\ref{prop:main-prior}, 
what we must prove is that for any 2-species network $G$ consisting of one reversible-reaction pair $y \lra y'$ and one irreversible reaction $\widetilde y \to \widetilde y'$, 
if $G$ is multistationary, then it is in fact nondegenerately multistationary.  Accordingly, let $G$ be such a network, and denote its species by $A$ and $B$.  We know by Proposition~\ref{prop:main-prior} that $y'-y = \lambda (\widetilde y' - \widetilde y)$ for some $0 \neq \lambda \in \mathbb{R}$, and also that 
 the embedded network of $G$ obtained by removing one of the species,
which without loss of generality we assume is species $B$,
 is a 2-alternating network (``$\lra \, \, \to$''
or
 ``$ \leftarrow \, \, \lra$'').
Thus, after switching $y$ and $y'$ if necessary (so that $y_A < y_A'$), we have that either
	\begin{align} \label{eq:arrow-cases}
	y_A < y_A'  < \widetilde y_A \quad \quad {\rm or} \quad \quad 
	\widetilde y_A < y_A < y_A'~,
	\end{align}
for, respectively, the ``$\lra \, \, \to$'' case %(that is, $y_A \lra y_A',~ \widetilde y_A \to \widetilde y_A' $)
or
the  ``$ \leftarrow \, \, \lra$'' case. % or ``$ \widetilde y_A' \leftarrow \widetilde y_A,~ y_A \lra y_A'$'' case.
%(Here we are also assuming, without loss of generality, that $y_A < y'_A$.

Each of these 2 cases breaks further into 6 subcases, based on:
	\begin{enumerate}
	\item whether the slope of the reaction vectors is positive (that is, $y_B < y_B'$) or negative ($y_B > y_B'$), and 
	\item whether 
		$y_B'< \widetilde y_B$, or
		$y_B'= \widetilde y_B$, or
		$y_B'> \widetilde y_B$; these three subcases correspond to when the boxes in the box diagrams look, respectively, as follows:
% ------ 3 box diagrams
\begin{center}
	\begin{tikzpicture}[scale=.5]
  %-------
  % box 1
  %-------
	\path [fill=gray] (-1,0) rectangle (0.5,1);
  % diagonal (NP)
	\draw (-1,0) --(0.5,1); 
  % labels
	\node [left] at (-1,0) {$y'$};
	\node [right] at (0.5, 1) {$\widetilde y$};
  %-------
  % box 2
  %-------
	\draw (3,0.5) -- (4.5,0.5);
  % labels
	\node [left] at (3,0.5) {$y'$};
	\node [right] at (4.5, 0.5) {$\widetilde y$};
  %-------
  % box 3
  %-------
	\path [fill=gray] (7,0) rectangle (8.5,1);
  % diagonal (NP)
	\draw (7,1) --(8.5,0); 
  % labels
	\node [left] at (7,1) {$y'$};
	\node [right] at (8.5,0) {$\widetilde y$};
	\end{tikzpicture}
\end{center}
	\end{enumerate}
(Regarding item (1) above, if $y_B = y_B'$, then $\frac{db}{dt}=0$, so this reduces to a 1-species network, and this case is done by Proposition~\ref{prop:main-prior}, part (1).)

We group the above possibilities as follows:

{\bf Case 1}: (a) 	$y_B > y_B' < \widetilde y_B$ or (b) $y_B < y_B' > \widetilde y_B$.  Visually, case (a) looks like one of the following, depending on which of the inequalities in~\eqref{eq:arrow-cases} holds:
% ------ 2 box diagrams
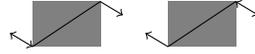
\begin{center}
	\begin{tikzpicture}[scale=.6]
  %-------
  % box 3
  %-------
	\path [fill=gray] (6,0) rectangle (7.5,1);
  % reactions
	\draw [<->] (6,0) -- (5.5,.3);
	\draw [->] (7.5,1) -- (8, .7);
  % diagonal (NP)
	\draw (6,0) --(7.5,1); 
  %-------
  % box 3 '
  %-------
	\path [fill=gray] (9,0) rectangle (10.5,1);
  % reactions
	\draw [->] (9,0) -- (8.5, .3);
	\draw [<->] (10.5,1) -- (11,0.7);
  % diagonal (NP)
	\draw (9,0) --(10.5,1);
	\end{tikzpicture}
\end{center}
Similarly, case (b) looks like one of the following diagrams:
% ------ 2 box diagrams
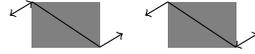
\begin{center}
	\begin{tikzpicture}[scale=.6]
 %-------
  % box 1
  %-------
	\path [fill=gray] (0,0) rectangle (1.5,1);
  % reactions
	\draw [<->] (0,1) -- (-0.5,0.7);
	\draw [->] (1.5,0) -- (2,.3);
  % diagonal (NP)
	\draw (0,1) --(1.5,0); 
  %-------
  % box 1'
  %-------
	\path [fill=gray] (3,0) rectangle (4.5,1);
  % reactions
	\draw [->] (3,1) -- (2.5,.7);
	\draw [<->] (4.5,0) -- (5,.3);
  % diagonal (NP)
	\draw (3,1) --(4.5,0); 
	\end{tikzpicture}
\end{center}

Thus, for such a network $G$, some subnetwork $N$ has the shape given in Proposition~\ref{prop:main-prior}, and thus $N$ is nondegenerately multistationary -- unless the slope of the marked diagonal is $-1$. So, when the slope of the marked diagonal is {\em not } $-1$, then by Lemma~\ref{lem:lift}, the original network $G$ also is nondegenerately multistationary.

Finally, we consider the subcase (of case (b)) in which the slope of the marked diagonal is $-1$, i.e., $y_A' + y_B'= \widetilde y_A + \widetilde y_B$.    Thus, the second and third summand in the right-hand side of~\eqref{eq:da_case_-1} below have the same total degree, and this degree is higher than that of the first summand.  The differential equations are:
	\begin{align} \label{eq:da_case_-1}
	\frac{da}{dt} ~&=~ \kappa_1(y_A'-y_A)a^{y_A} b^{y_B} 
					- \kappa_2 (y_A' -y_A) a^{y_A'} b^{y'_B}
					+ \kappa_3 (\widetilde y_A' - \widetilde y_A) a^{\widetilde y_A} b^{\widetilde y_B} \\
	\frac{db}{dt} ~&=  \mu \frac{da}{dt}~, \notag
	\end{align}
where $\mu := (y_B'-y_B)/(y'_A-y_A) >0$.  Hence, we are interested in counting the number of positive multiplicity-one  roots of the right-hand side of~\eqref{eq:da_case_-1}, when the substitution $b:=\mu a +T$ is made, 
and we are free to choose any real value for $T$ and any positive values for the $\kappa_i$'s.   After performing the following operations to the right-hand side of~\eqref{eq:da_case_-1}: 
\begin{enumerate}
\item Substitute $b:=  \mu a +T$, and
\item Divide by $a^{y_A}$ (which is fine because we are interested in positive roots).
\end{enumerate}
we obtain:
\begin{align} \label{eq:g_case-1}
	g(a) ~&:=~ \kappa_1(y_A'-y_A) ( \mu a +T)^{y_B} 
					- \kappa_2 (y_A' -y_A) a^{y_A'-y_A} ( \mu a +T)^{y'_B}
					\\ \notag
				& \quad \quad	+ \kappa_3 (\widetilde y_A' - \widetilde y_A) a^{\widetilde y_A-y_A} ( \mu a +T)^{\widetilde y_B} ~.
\end{align}
We can choose $\kappa_2$ and $\kappa_3$ so that the leading coefficient of $g$ is positive (by ensuring that the inequality 
$\kappa_2 (y_A' -y_A) \mu^{y_B'} <
\kappa_3 (\widetilde y_A' - \widetilde y_A) \mu^{\widetilde y_B}$ holds), 
so $\lim_{a \to \infty} g(a)= \infty$.  Also, notice that $g(0)>0$ as long as $T>0$.  So, by the intermediate value theorem, 
it suffices to show that $g(1)<0$ when $T$ and $\kappa_1$ are chosen appropriately.  To see this, observe:
\begin{align*}% \label{eq:g_case-1}
	g(1) ~&=~ \kappa_1(y_A'-y_A) ( \mu +T)^{y_B} 
					- \kappa_2 (y_A' -y_A) ( \mu +T)^{y'_B}
					+ \kappa_3 (\widetilde y_A' - \widetilde y_A)  ( \mu +T)^{\widetilde y_B} ~,
\end{align*}
and recall that $y_B < y'_B> \widetilde y_B$, so for $T$ sufficiently large, $g(1)<0$.

%\noindent
{\bf Case 2}: $y_B < y_B' \leq \widetilde y_B$.
There are, from~\eqref{eq:arrow-cases}, two subcases. We consider first the subcase of $y_A < y_A'  < \widetilde y_A$ (``$\lra \, \, \to$''), depicted here:
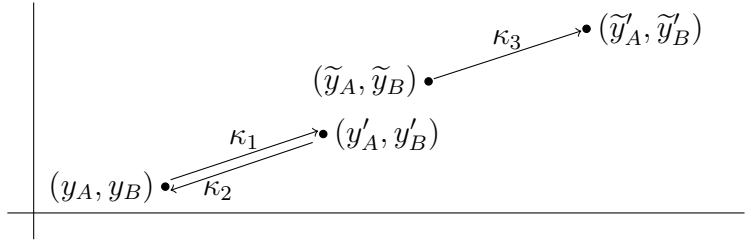
\begin{center}
	\begin{tikzpicture}[scale=.7]
]% axes
	\draw (-2,-0.5) -- (12,-0.5);
	\draw (-1.5,-1) -- (-1.5,3.5);
% reactions
	\draw [->] (1.1,0.13) -- (3.9,1.06);
	\draw [<-] (1.1,-0.07) -- (3.8,0.84);
	\draw [->] (6.1,2.05) -- (8.9,2.96);
% points
	\filldraw [black] (1,0) circle (2pt);
	\filldraw [black] (4,1) circle (2pt);
	\filldraw [black] (6,2) circle (2pt);
	\filldraw [black] (9,3) circle (2pt);
% labels
    \node [left] at (1,0) {$(y_A,y_B)$};
    \node [right] at (4,1) {$(y'_A,y'_B)$};
    \node [left] at (6,2) {$(\widetilde y_A, \widetilde y_B)$};
    \node [right] at (9,3) {$(\widetilde y'_A, \widetilde y'_B)$};
% rate constants
    \node [right] at (2,0.9) {$\kappa_1$};
    \node [below] at (2,0.3) {$\kappa_2$};
    \node [right] at (7,2.8) {$\kappa_3$};
	\end{tikzpicture}
\end{center}

Hence,
	\begin{align} \label{eq:da}
	\frac{da}{dt} ~&=~ \kappa_1(y_A'-y_A)a^{y_A} b^{y_B} 
					- \kappa_2 (y_A' -y_A) a^{y_A'} b^{y'_B}
					+ \kappa_3 (\widetilde y_A' - \widetilde y_A) a^{\widetilde y_A} b^{\widetilde y_B} \\
	\frac{db}{dt} ~&=  \mu \frac{da}{dt}~,  \notag
	\end{align}
where $\mu := (y_B'-y_B)/(y'_A-y_A) >0$.  Hence, we are interested in counting the number of positive multiplicity-one  roots of the right-hand side of~\eqref{eq:da}, when the substitution $b:=T+\mu a$ is made, and 
we are free to choose any real value for $T$ and any positive values for the $\kappa_i$'s.  
Let $p_1 := y_A' - y_A$ and $p_2 := \widetilde y_A- y_A$ (so, the $p_i$'s are integers satisfying $1 \leq p_1 \leq p_2$), and let $n_1:= y'_B -  y_B$ and $n_2:= \widetilde y_B- y_B$ (so the $n_i$'s are integers with $0 \leq n_1 < n_2$).  After performing the following three operations on the right-hand side of~\eqref{eq:da}: 
\begin{enumerate}
\item Divide by $a^{y_A}b^{ y_B}$ (which is fine because we are interested in positive roots), \item Substitute $b:= \mu a$ (that is, we pick $T=0$),
and 
\item Divide by the positive term $\kappa_1(y_A'-y_A)$,
\end{enumerate}
we obtain:
\begin{align} \notag
	g(a) ~& :=~ 1 - \frac{\kappa_2}{\kappa_1} a^{p_1}(\mu a)^{n_1} + 
			 \frac{\kappa_3 (\widetilde y_A' - \widetilde y_A)}{\kappa_1(y_A'-y_A)} a^{p_2} (\mu a)^{n_2} \\
		~&~=~ 1 - l a^{p_1+n_1}  + m a^{p_2+n_2},	 \label{eq:g-case-2}
\end{align}
where $l:={\kappa_2}/{\kappa_1} \mu^{n_1}$ and $m:={\kappa_3 (\widetilde y_A' - \widetilde y_A)}/({\kappa_1(y_A'-y_A)} ) \mu^{n_2}$.  Note that $p_1+n_1 < p_2 +n_2$.
Also, we can choose any positive values for $l$ and $m$ by choosing the (positive) $\kappa_i$'s appropriately.
Thus, our question is whether there exist positive values of $m$ and $l$ for which the univariate polynomial $g(a)$, in~\eqref{eq:g-case-2}, admits two more positive multiplicity-one roots.  Indeed, this follows from the converse of Descartes' rule of signs~\cite[Theorem 1]{grabiner}, restated in \cite[Lemma 3.16]{Joshi:Shiu:Multistationary}.

Finally, the remaining subcase, when $\widetilde y_A < y_A < y_A'$ (the ``$\leftarrow \lra$'' case), is similar.  Specifically, after performing the steps analogous to those for the prior subcase, we obtain a polynomial whose negative has the form equal to the expression in~\eqref{eq:g-case-2}.  So, again, we can use the converse of Descartes' rule of signs to complete this subcase.

%\noindent
{\bf Remaining case}:	$y_B > y'_B \geq \widetilde y_B$.
%	Hence, the only case left is when the following inequalities hold: $y_B > y'_B$ and $y'_B \geq \widetilde y_B$.  
There are, from~\eqref{eq:arrow-cases}, two subcases. We consider first the subcase of $y_A < y_A'  < \widetilde y_A$ (``$\lra \, \, \to$''), depicted here:
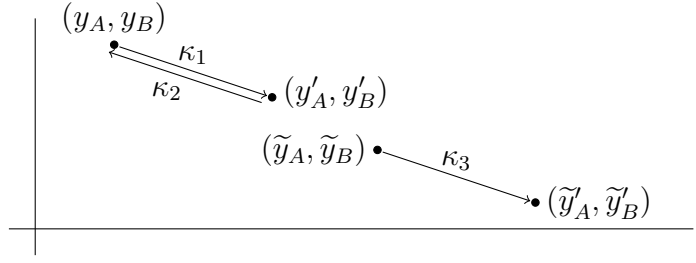
\begin{center}
	\begin{tikzpicture}[scale=.7]
   % axes
	\draw (-1,-0.5) -- (12,-0.5);
	\draw (-0.5,-1) -- (-0.5,3.5);
%    \node [right] at (9,-0.5) {$A$};
  % reactions
	\draw [->] (1.1,2.96) -- (3.9,2.03);
	\draw [<-] (0.9,2.86) -- (3.8,1.9);
	\draw [->] (6.1,0.96) -- (8.9,0.03);
% points
	\filldraw [black] (1,3) circle (2pt);
	\filldraw [black] (4,2) circle (2pt);
	\filldraw [black] (6,1) circle (2pt);
	\filldraw [black] (9,0) circle (2pt);
    % labels
    \node [above] at (1,3) {$(y_A,y_B)$};
    \node [right] at (4,2.1) {$(y'_A,y'_B)$};
    \node [left] at (6,1) {$(\widetilde y_A, \widetilde y_B)$};
    \node [right] at (9,0) {$(\widetilde y'_A, \widetilde y'_B)$};
% rate constants
    \node [right] at (2,2.8) {$\kappa_1$};
    \node [below] at (2,2.5) {$\kappa_2$};
    \node [right] at (7,0.8) {$\kappa_3$};
	\end{tikzpicture}
\end{center}

Hence,
	\begin{align} \label{eq:da-again}
	\frac{da}{dt} ~&=~ \kappa_1(y_A'-y_A)a^{y_A} b^{y_B} 
					- \kappa_2 (y_A' -y_A) a^{y_A'} b^{y'_B}
					+ \kappa_3 (\widetilde y_A' - \widetilde y_A) a^{\widetilde y_A} b^{\widetilde y_B} \\
	\frac{db}{dt} ~&= - \mu \frac{da}{dt}~,  \notag
	\end{align}
where $\mu := (y_B-y_B')/(y'_A-y_A) >0$.  Hence, we are interested in counting the number of positive multiplicity-one  roots of the right-hand side of~\eqref{eq:da-again}, when the substitution $b:=T-\mu a$ is made, and we are free to choose any positive values of $T$ and the $\kappa_i$'s.
Let $p_1 := y_A' - y_A$ and $p_2 := \widetilde y_A- y_A$ (so, the $p_i$'s are integers satisfying $1 \leq p_1 \leq p_2$), and 
let $n_1:= y'_B - \widetilde y_B$ and $n_2:= y_B- \widetilde y_B$ (so the $n_i$'s are integers with $0 \leq n_1 < n_2$).  After performing the following three operations to the right-hand side of~\eqref{eq:da-again}: 
\begin{enumerate}
\item Divide by $a^{y_A}b^{\widetilde y_B}$ (which is fine because we are interested in positive roots), \item Substitute $b:= T- \mu a$,
and 
\item Divide by the positive term $\kappa_1(y_A'-y_A)$,
\end{enumerate}
we obtain:
\begin{align} \notag
	g(a) ~&:=~(T-\mu a)^{n_2}
			- \frac{\kappa_2 (y_A' -y_A) }{\kappa_1(y_A'-y_A)} a^{p_1} (T-\mu a)^{n_1}
			+ \frac{\kappa_3 (\widetilde y_A' - \widetilde y_A)}{\kappa_1(y_A'-y_A)} a^{p_2} \\
		~&~= (T-\mu a)^{n_2} - l a^{p_1} (T- \mu a)^{n_1} + m a^{p_2},	 \label{eq:g}
\end{align}
where $l:={\kappa_2}/{\kappa_1}$ and $m:={\kappa_3 (\widetilde y_A' - \widetilde y_A)}/({\kappa_1(y_A'-y_A)} )$.  Note that we can choose any positive values for $l$ and $m$ by choosing the (positive) $\kappa_i$'s appropriately.

Thus, our question is whether there exist positive values of $T, m, l$ for which the univariate polynomial $g(a)$, in~\eqref{eq:g}, admits two more positive multiplicity-one roots in the interval $(0, T/\mu)$. 
 We already know, because $G$ is multistationary, that $g$ admits two or more distinct positive roots in such an interval (for some choice of positive $T,m,l$).  Thus, by Proposition~\ref{Prop:PolyTwoDistincPositiveRoots}, we get the desired conclusion.

Finally, the remaining subcase, when $\widetilde y_A < y_A < y_A'$ (the ``$\leftarrow \lra$'' case), is similar.  Specifically, after performing the steps analogous to those for the prior subcase, we obtain a polynomial whose negative has the form equal to the expression in~\eqref{eq:g}.  So, again, we can use Proposition~\ref{Prop:PolyTwoDistincPositiveRoots} to complete the proof.
\end{proof}

\section{Toward results for three or more species} \label{sec:3+}
In this section, we describe efforts toward extending Theorem~\ref{thm:main} to allow for more than two species.  Specifically, 
our future goal is to prove the following conjecture: 

\begin{conjecture} \label{conj:2-species}
A network $G$ that consists of one reversible-reaction pair $y \lra y'$
and one irreversible reaction $\widetilde y \to \widetilde y'$ 
 is nondegenerately multistationary
 if and only if 
 the reaction vectors are (nontrivial) scalar multiples of each other: 
 $y'-y = \lambda (\widetilde y' - \widetilde y)$ for some $0 \neq \lambda \in \mathbb{R}$, and, for some species $i$, the embedded network of $G$ obtained by removing all species except $i$ is a 2-alternating network (``$\lra \, \, \to$''
or
 ``$ \leftarrow \, \, \lra$'').%
\end{conjecture}
We unfortunately
cannot prove Conjecture~\ref{conj:2-species}, but in some cases 
(see Example~\ref{ex:works})
but not all
(Example~\ref{ex:fails})
we can reduce networks with 3 or more species to the case of 2 species.

%--------------------
% EXAMPLE THAT CAN REDUCE TO 2 SPECIES
%--------------------
\begin{example} \label{ex:works}
Consider the following 4-species network:
	\begin{align} \label{eq:network-works}
			2C+2D
			 \mathrel{\mathop{\rightleftarrows}^{k_1}_{k_2}}
			 A+B+C+D
			  \quad \quad 
			  2A+2B+C+D 
		\overset{k_3}\to 3A+3B~.
	\end{align}
The conservation-law equations are
	\begin{align} \label{eq:cons-law}
	b=a+T_1 \quad \quad 
	c= T_2 - a \quad \quad
	d=T_3-a~,
	\end{align}
for some $T_1 \in \mathbb{R}$ and $T_2,T_3>0$.  
After substituting the equations~\eqref{eq:cons-law} into the steady-state equation, we obtain:
	\begin{align} \label{eq:after-sub}
	0 ~=~ k_1 (T_2-a)^2 (T_3-a)^2
		- k_2 a (a+T_1)(T_2-a)(T_3-a)
		+ k_3 a^2 (a+T_1)^2 (T_2-a) (T_3-a)~.
	\end{align}	
If we choose $T_1=0$ and $T_2=T_3=:T$, equation~\eqref{eq:after-sub}
reduces to:
	\begin{align} \label{eq:now-reduced}
	0 ~=~ k_1 (T-a)^4
		- k_2 a^2  (T-a)^2
		+ k_3 a^4  (T-a)^2~,
	\end{align}	
which in turn has the general form of the 
steady-state equation
(after conservation-law substitution) for the following network:
	\begin{align} \label{eq:network-works-reduced}
			4E
			 \mathrel{\mathop{\rightleftarrows}^{k_1}_{k_2}}
			 2A+2E
			  \quad \quad 
			  4A+2E 
		\overset{k_3}\to 6A~.
	\end{align}
Network~\eqref{eq:network-works-reduced} is known from Theorem~\ref{thm:main}
to be nondegenerately multistationary, so there exists $T>0$ such that equation~\eqref{eq:now-reduced}
 has multiple nondegenerate roots.  Therefore, the original network~\eqref{eq:network-works} is also nondegenerately multistationary.
\end{example}

In Example~\ref{ex:works}, we showed that the 4-species network~\eqref{eq:network-works} is
nondegenerately multistationary 
by reducing to the 2-species case.  Let us summarize this approach, which applies to
certain networks (with $3$ or more species) in which every species $i=2, 3, \dots, s$
satisfies $\frac{dx_i}{dt}=  \pm \frac{dx_1}{dt}$
(here, without loss of generality, 
the species having the form
``$\lra \, \, \to$''
or
 ``$ \leftarrow \, \, \lra$'' is species 1).
For every species $i=2,...,s$ for which $\frac{dx_i}{dt}= \frac{dx_1}{dt}$, 
we set
 $T_i=0$, and then we set all remaining $T_i$'s equal to each other.  
If the resulting steady-state equation
has the form arising from a (2-species) network 
that is known to be nondegenerately multistationary, then we are done: the original network also is.  

This technique, however, does not always work, as the following example shows.

%--------------------
% EXAMPLE THAT CAN not REDUCE TO 2 SPECIES
%--------------------
\begin{example} \label{ex:fails}
Consider the following 4-species network:
	\begin{align} \label{eq:network-fails}
			2C+2D
			 \mathrel{\mathop{\rightleftarrows}^{k_1}_{k_2}}
			 A+B+C+D
			  \quad \quad 
			  2A+C+D 
		\overset{k_3}\to 3A+B~.
	\end{align}
The conservation-law equations are 
 given in \eqref{eq:cons-law}, the same as those for 
 Example~\ref{ex:works}.
 After substituting the equations~\eqref{eq:cons-law} into the steady-state equation, we obtain:
	\begin{align} \label{eq:after-sub-fails}
	0 ~=~ k_1 (T_2-a)^2 (T_3-a)^2
		- k_2 a (a+T_1)(T_2-a)(T_3-a)
		+ k_3 a^2  (T_2-a) (T_3-a)~.
	\end{align}	
This time, however, when we choose $T_1=0$ and $T_2=T_3=:T$, equation~\eqref{eq:after-sub-fails} becomes:
	\begin{align*} %\label{eq:now-reduced}
	0 ~&=~ k_1 (T-a)^4
		- k_2 a^2  (T-a)^2
		+ k_3 a^2  (T-a)^2 \\
	~&=~	 k_1 (T-a)^4
		+ (-k_2+k_3) a^2  (T-a)^2~,
	\end{align*}	
which does {\em not} arise from the steady-state equation of a 2-species network 
that is known to be nondegenerate.  Hence, if we want to show that network~\eqref{eq:network-fails} is nondegenerately multistationary, we will need another approach.  (For such an approach, see Remark~\ref{rmk:alternate}.)
\end{example}

\begin{remark} \label{rmk:alternate}
An ad-hoc method for proving that network~\eqref{eq:network-fails} is nondegenerately multistationary is as follows.  First, rearrange \eqref{eq:after-sub-fails} as follows:
	\begin{align} \label{eq:after-sub-fails-simplified}
	0 ~=~ (T_2-a) (T_3-a) \cdot \left[
		  (k_1 + k_3 - k_2) a^2 			
		- (k_1(T_2 + T_3) + k_2 T_1) a
		+ k_1 T_2 T_3
		\right]~.
	\end{align}
Now choose $(k_1,k_2,k_3):= (2/9,1,16/9)$
and
$(T_1,T_2,T_3):=(8/3,3,3)$
so that~\eqref{eq:after-sub-fails-simplified}
becomes:
	\begin{align*}
	0 ~=~ (T_2-a) (T_3-a) \left[
		 a^2 			
		- 3 a
		+ 2		
		\right]~=~
		(3-a) (3-a)(a-2)(a-1)~.
	\end{align*}
This equation has two simple roots, $a^*=1$ and $a^*=2$, in the interval
$\left(0, {\rm min}(T_2,T_3) \right)=(0,3)$.  These roots correspond to nondegenerate steady states, so network~\eqref{eq:network-fails} is nondegenerately multistationary.
\end{remark}

\section{Discussion} \label{sec:disc}
Our work was motivated by the Nondegeneracy Conjecture: Is a network multistationary if and only if it is {\em nondegenerately} multistationary?  
At first, one might think this is easily so; we would expect to be able to perturb parameters to make a degenerate steady state become nondegenerate.  
Indeed, we succeed in doing precisely this for small networks (Theorems~\ref{thm:main} and~\ref{thm:main2}).
Nevertheless, such arguments are subtle. The perturbations must be done carefully, as we saw in the proof of Lemma~\ref{Lem:ReduceMultiplicityOfRoots}.

Looking forward, 
we expect that the algebraic techniques we used here 
will help us classify 
 more (perhaps all) one-dimensional reaction systems
  (recall Conjecture~\ref{conj:2-species} and see also \cite[Question 6.1]{Joshi:Shiu:Multistationary}).  
  Indeed, to resolve such problems, we 
  will need tools for analyzing families of univariate polynomials.

Finally, as mentioned earlier, our true interest in applications goes beyond multistationarity -- to {\em multistability}.  We do not yet have a complete classification of one-dimensional multistable networks, not even among networks consisting of (a) 2 irreversible reactions, (b) 1 irreversible and 1 reversible-reaction pair, or (c) 2 reversible-reaction pairs (\cite[Question 6.2]{Joshi:Shiu:Multistationary}).  What our work contributes here are corresponding results at the level of multistationarity -- which then point the way forward for achieving multistability.

{\small
\subsection*{Acknowledgements}
We thank two referees for their detailed comments, which improved this work. 
TdW was partially supported by the DFG (WO 2206/1-1).  
AS was partially supported by the NSF (DMS-1312473/DMS-1513364)
and the Simons Foundation (\#521874).  
This article was finalized while TdW was hosted by the Institut Mittag-Leffler. We thank the institute for its hospitality.
}

\bibliographystyle{amsalpha}
\bibliography{Shiu_deWolff_SmallReactionNetworks}

\providecommand{\bysame}{\leavevmode\hbox to3em{\hrulefill}\thinspace}
\providecommand{\MR}{\relax\ifhmode\unskip\space\fi MR }
% \MRhref is called by the amsart/book/proc definition of \MR.
\providecommand{\MRhref}[2]{%
  \href{http://www.ams.org/mathscinet-getitem?mr=#1}{#2}
}
\providecommand{\href}[2]{#2}
\begin{thebibliography}{CFMW17}

\bibitem[BP16]{BP-inher}
Murad Banaji and Casian Pantea, \emph{The inheritance of nondegenerate
  multistationarity in chemical reaction networks}, preprint, {\tt
  arXiv:1608.08400}. (2016).

\bibitem[CA00]{HowSwitch}
Joshua~L. Cherry and Frederick~R. Adler, \emph{How to make a biological
  switch}, J.\ Theoret.\ Biol. \textbf{203} (2000), no.~2, 117--133.

\bibitem[CFMW17]{CFMW}
Carsten Conradi, Elisenda Feliu, Maya Mincheva, and Carsten Wiuf,
  \emph{Identifying parameter regions for multistationarity}, PLoS Comput.\
  Biol. \textbf{13} (2017), no.~10, e1005751.

\bibitem[CS18]{perspective}
Carsten Conradi and Anne Shiu, \emph{Dynamics of post-translational
  modification systems: recent progress and future challenges}, Biophys.\ J.,
  to appear (2018).

\bibitem[DDG15]{dexter2015}
Joseph~P. Dexter, Tathagata Dasgupta, and Jeremy Gunawardena, \emph{Invariants
  reveal multiple forms of robustness in bifunctional enzyme systems}, Integr.\
  Biol. \textbf{7} (2015), 883--894.

\bibitem[Dic16]{Dickenstein:Invitation}
Alicia Dickenstein, \emph{Biochemical reaction networks: an invitation for
  algebraic geometers}, Mathematical {C}ongress of the {A}mericas, Contemp.
  Math., vol. 656, Amer. Math. Soc., Providence, RI, 2016, pp.~65--83.

\bibitem[Fei87]{FeinDefZeroOne}
Martin Feinberg, \emph{{Chemical reaction network structure and the stability
  of complex isothermal reactors I. The deficiency zero and deficiency one
  theorems}}, Chem. Eng. Sci. \textbf{42} (1987), no.~10, 2229--2268.

\bibitem[FSW16]{FSW}
Bryan F\'elix, Anne Shiu, and Zev Woodstock, \emph{Analyzing multistationarity
  in chemical reaction networks using the determinant optimization method},
  Appl.\ Math.\ Comput. \textbf{287--288} (2016), 60--73.

\bibitem[GH02]{Karin02}
Karin Gatermann and Birkett Huber, \emph{A family of sparse polynomial systems
  arising in chemical reaction systems}, J.\ Symb.\ Comput. \textbf{33} (2002),
  no.~3, 275--305. \MR{MR1882230 (2003b:92031)}

\bibitem[GHRS16]{case-study}
Elizabeth Gross, Heather~A. Harrington, Zvi Rosen, and Bernd Sturmfels,
  \emph{Algebraic systems biology: a case study for the {Wnt} pathway}, Bull.
  Math. Biol. \textbf{78} (2016), no.~1, 21--51.

\bibitem[GKZ94]{Gelfand:Kapranov:Zelevinsky}
Izrailʹ~M. Gelʹfand, Mikhail~M. Kapranov, and Andrey~V. Zelevinsky,
  \emph{Discriminants, resultants and multidimensional determinants},
  Birkh{\"a}user, 1994.

\bibitem[GMS14]{GMS2}
Manoj Gopalkrishnan, Ezra Miller, and Anne Shiu, \emph{A geometric approach to
  the global attractor conjecture}, SIAM J. Appl. Dyn. Syst. \textbf{13}
  (2014), no.~2, 758--797.

\bibitem[Gra99]{grabiner}
David~J. Grabiner, \emph{Descartes' rule of signs: another construction},
  Amer.\ Math.\ Monthly \textbf{106} (1999), no.~9, 854--856. \MR{1732666
  (2000i:12001)}

\bibitem[HH10]{HoHarrington}
Kenneth~L. Ho and Heather~A. Harrington, \emph{Bistability in apoptosis by
  receptor clustering}, PLoS Comput.\ Biol. \textbf{6} (2010), no.~10,
  e1000956.

\bibitem[JS13]{Joshi:Shiu:Atoms}
Badal Joshi and Anne Shiu, \emph{Atoms of multistationarity in chemical
  reaction networks}, J.\ Math.\ Chem. \textbf{51} (2013), no.~1, 153--178.

\bibitem[JS15]{mss-review}
\bysame, \emph{{A} survey of methods for deciding whether a reaction network is
  multistationary}, Math. Model. Nat. Phenom., special issue on ``Chemical
  dynamics'' \textbf{10} (2015), no.~5, 47--67.

\bibitem[JS17]{Joshi:Shiu:Multistationary}
\bysame, \emph{Which small reaction networks are multistationary?}, SIAM J.\
  Appl.\ Dyn.\ Syst. \textbf{16} (2017), no.~2, 802--833.

\bibitem[MD16]{messi}
Mercedes~P{\'e}rez Mill{\'a}n and Alicia Dickenstein, \emph{The structure of
  {MESSI} biological systems}, preprint, {\tt arXiv:1612.08763} (2016).

\bibitem[RS02]{Rahman:Schmeisser:AnalyticTheoryOfPolynomials}
Qazi~I. Rahman and Gerhard Schmeisser, \emph{Analytic theory of polynomials},
  London Mathematical Society Monographs. New Series, vol.~26, The Clarendon
  Press, Oxford University Press, Oxford, 2002.

\bibitem[Swe17]{sweeney}
Mark~A Sweeney, \emph{Conditions for solvability in chemical reaction networks
  at quasi-steady-state}, Preprint, {\tt arXiv:1712.05533} (2017).

\bibitem[TdW16]{Theobald:deWolff:Trinomials}
Thorsten Theobald and Timo de~Wolff, \emph{Norms of roots of trinomials}, Math.
  Ann. \textbf{366} (2016), no.~1-2, 219--247.

\bibitem[Vas92]{Vassiliev:Discriminants}
Victor~A. Vasilʹev, \emph{Complements of discriminants of smooth maps:
  topology and applications}, Translations of Mathematical Monographs, vol.~98,
  American Mathematical Society, Providence, RI, 1992, Translated from the
  Russian by B. Goldfarb.

\end{thebibliography}

\end{document}